\documentclass[12pt,leqno]{amsart}
\author{No\'{e} B\'{a}rcenas }
                \email{barcenas@matmor.unam.mx}
         \urladdr{http://www.matmor.unam.mx /~ barcenas}

 \address{Centro de Ciencias Matem\'aticas. UNAM \\Ap.Postal 61-3 Xangari. Morelia, Michoac\'an MEXICO 58089}
 
\author{Mario Vel\'asquez}
\address{Universidad Nacional de Colombia Sede Bogot\'a Departamento de Matem\'aticas, Facultad de Ciencias Cra. 30 cll 45 - Ciudad Universitaria, Bogot\'a, Colombia} \email{mavelasquezme@unal.edu.co}  
\urladdr{https://sites.google.com/site/mavelasquezm/home}

\usepackage{graphicx}
\usepackage{hyperref}
\usepackage{pinlabel} 
\usepackage{amsmath}
\usepackage{amsfonts}
\usepackage{amssymb}
\usepackage{mathtools}
\usepackage{amscd}
\usepackage[all,cmtip]{xy}
\usepackage{amsthm}
\usepackage{comment}
\usepackage{epsfig}
\usepackage[utf8]{inputenc}
\usepackage[colorinlistoftodos]{todonotes}
\usepackage[capitalise]{cleveref}
\usepackage{geometry}
 \usepackage{url} 
 \usepackage{enumerate}

\makeatletter
\providecommand\@dotsep{5}
\def\listtodoname{List of Todos}
\def\listoftodos{\@starttoc{tdo}\listtodoname}
\makeatother

\newtheorem{theorem}{Theorem}[section]

\newtheorem{corollary}[theorem]{Corollary}
\newtheorem{lemma}[theorem]{Lemma}

\newtheorem*{theorem*}{Theorem}

  \theoremstyle{definition}

\newtheorem{condition}[theorem]{Condition}

\newcommand{\dbR}{\mathbb{R}}

\newcommand{\dbZ}{\mathbb{Z}}

\newcommand{\Z}{\mathbb{Z}}

\newcommand{\calF}{{\mathcal F}}

\newcommand{\calG}{{\mathcal G}}
\newcommand{\calH}{{\mathcal H}}

\newcommand{\calM}{{\mathcal M}}

\newcommand{\beq}{\begin{equation}}
\newcommand{\eeq}{\end{equation}}

\title{Positive Scalar Curvature and crystallographic fundamental groups.}

\begin{document}

\maketitle

\begin{abstract}
    We  examine positive  and  negative results for  the  Gromov-Lawson-Rosenberg Conjecture within  the class of  crystallographic groups.  We give necessary conditions  within  the  class  of  split extensions  of free abelian  by cyclic  groups to satisfy the unstable Gromov-Lawson-Rosenberg Conjecture. We also give necessary conditions within the  same class  of  groups which are counterexamples for the conjecture. 
\end{abstract}

\section{Introduction}

The  (unstable) Gromov-Lawson-Rosenberg Conjecture for a discrete group $\Gamma$  predicts that a  closed spin $n$-dimensional manifold $M^{n}$ where  $n\geq 5$, with fundamental group $\Gamma$ and classifying  map for  the  fundamental group  $f: M\to B\Gamma $,  the vanishing of the   group homomorphism 

$$ \alpha (M) = A\circ p_{B\Gamma} D(f_{M} ):\Omega_{n}^{\rm spin}(B\Gamma)\longrightarrow KO_{n}(C_{*}^{*}(\Gamma)) $$
given as the composition
$$ \Omega_{n}^{\rm spin}(B\Gamma)\overset{D(f_{M}) }{\longrightarrow} ko_{n}(B\Gamma)\overset{p_{B\Gamma}}{\longrightarrow} KO_{n}(B\Gamma)\overset{A}{\longrightarrow} KO_{n}(C_{r}^{*}(\Gamma)), $$
decides about  the  existence  of  a  metric  of  positive  scalar  curvature  on $M$. 
Some  explanations  are due.
$D$ is  the map  which  sends  a spin  bordism  class $f_{M}:M\to B\Gamma$   to  the  image  of  the  $ko$-fundamental  class $f_{*}([M])\in ko_{n}(B\Gamma)$. The map 
$p_{B\Gamma}: ko_{n}(B\Gamma)\to KO_{n}(B\Gamma)$ is  the  natural transformation  of  periodicity, and $A$  denotes  the real  assembly  map $A:KO_{n}(B\Gamma)\to KO_{n}(C_{r}^{*}(\Gamma)) $. 

There  exist  counterexamples  to this  conjecture. T. Schick  in  \cite{schick} showed  that for  the group $\Gamma= \mathbb{Z}^{4}\times \mathbb{Z}/3$, there  exists a five dimensional  manifold $M$ with  fundamental  group $\Gamma= \pi_{1}(M)=  \mathbb{Z}^{4}\times \mathbb{Z}/3$ for which
$$ \alpha(M)=0 \in KO_{5}(C_{r}^{*}(\Gamma)),$$
but $M$ admits no metric of  positive  scalar curvature. 

The result initiated a  series  of subsecuent  articles stating group cohomological  conditions  which  produce  counterexamples  for  the (unstable)  Gromov-Lawson-Rosenberg Conjecture, including \cite{davispearson}, and specially \cite{DSS},  where the techniques  are  used  to  construct  a torsionfree example, which is even  a  fundamental group of a compact manifold  admitting  a ${\rm CAT
}(0)$-cubical complex structure.

 We  would  like  to mention \cite{DP}, \cite{DL13} and \cite{HU} as  some  sources  for  positive  results  on the  stable  Gromov-Lawson-Rosenberg-Conjecture. The  positive results  therein  concern  groups satisfying  the  Baum Connes Isomorphism  conjecture,  which  satisfy  condition \ref{condition:NM},  and some other conditions about the maximal finite subgroups. 
 
We  will consider  in this article for a  group  homomorphism $\rho: \mathbb{Z}/m\to Gl_{n}(\mathbb{Z})$, split  extensions  of  the  type
$$ 1\longrightarrow \mathbb{Z}^{n}\longrightarrow \Gamma = \mathbb{Z}^{n}\rtimes_{\rho} \mathbb{Z}/m\longrightarrow \mathbb{Z}/m\longrightarrow 1.$$

The integral cohomology of  such  groups  $\Gamma$ has  been  computed  in  a  series  of  articles under several additional  sets  of  hypothesis including: 
\begin{itemize}
\item The group $\Gamma$ is  torsionfree and $m$ is  a  prime  number \cite{CV}. 
\item The  action is   compatible in the  sense of \cite{AGPP}, which  allows  for a  specific  resolution   of  the  trivial $\mathbb{Z}[\Gamma]$-module  $\mathbb{Z}$,  and  the  collapse  of the Lyndon-Hochschild-Serre  spectral sequence computing  the  integral cohomology of $\Gamma $  without  extension problems  at  the  $E_{2}$-term.  

\item The action of $\Gamma$ on $\mathbb{R}^{n}-\{0\}$ is  free outside of  the  origin. \cite{LL12}. 
\item The action  of  $\Gamma$ on $\mathbb{R}^{n}-\{0\}$  is  free outside  of the  origin,   and  $m$ is  a  prime  number \cite{DL13}. 

\item The natural number $m$
 is  free of  squares, without  further assumption  on the  action \cite{SV}.  
 \end{itemize}
 With  the  exception of \cite{SV} and \cite{AGPP}, these  conditions  are used  because they imply  the  following maximality properties within the family  of finite  subgroups  of $\Gamma$. 
 
\begin{condition}\label{condition:NM}
[Conditions $M$ and $NM$]
\begin{itemize}
\item  Each finite  subgroup $H$ of $\Gamma$ is a subgroup of  a unique maximal finite  subgroup  $M$, and  there  exists  a  finite  collection (up to conjugacy) $\mathcal{M}$ of  them. 
\item  The  normalizer in $\Gamma$  of $M$ is  $M$ itself. 
\end{itemize}
\end{condition}

L\"uck and Davis in \cite{DL13} used  the  results  of  these computations  together  with  the  construction  of  specific  models  for  the  classifying space for  proper  actions \cite{LW} to  derive computations of  complex, real, and  real  connective  $K$-homology of  both  the  classifying  space $B\Gamma$,  and  the  classifying  space  for  proper  actions $\underline{E}\Gamma$.

Extending  these results, the second  named author  and S\'anchez performed  computations  of both  the complex $KU$- homology  of  the  classifying  spaces $B\Gamma$,  denoted by  $KU_{*}(B\Gamma)$, and  the  equivariant $KU$-homology groups  of  the  classifying  spaces  for proper  actions,  denoted  by  $KU_{*}^{\Gamma}(\underline{E}\Gamma )$. 

In this work we  will  make  structural statements  about  the  algebraic  structure of real connective $ko$-homology  groups  of $B \Gamma$, denoted  by  $ko_{*}(B\Gamma)$, which  will be  the  base for positive and  negative results for the (unstable) Gromov-Lawson-Rosenberg conjecture for high dimensional smooth  spin manifolds  with fundamental group $\Gamma$. 

The  hypothesis  that  we  will  impose  on the  group  $\Gamma$ is  the  following

\begin{condition}\label{condition:positive}[Condition for  positive results]
Let $m$ be  and odd natural number  and  assume  that $\rho: \mathbb{Z}/m\to Gl_{n}(\mathbb{Z})$ is  a  group homomorphism such that the group action  of  $\mathbb{Z}/M$ on $\mathbb{R}^{n}-\{0\}$  is  free. 
\end{condition}

The  following  is  our  main positive  result  on the  Gromov-Lawson-Rosenberg  conjecture. 

\begin{theorem}\label{theo:global}

Let $M^n$ be an $n$-dimensional smooth spin manifold, where $n\geq 5$ is  even, and  with fundamental group isomorphic to $\Gamma$, where $\Gamma$ satisfies condition \ref{condition:positive}.  Denote  by $f_{M}: M\to B\Gamma $ the  classifying map for  the  fundamental group. Assume  that $\alpha(M) =0$. Then $M$ admits a  metric of  positive scalar curvature.
\end{theorem}

Theorem \ref{theo:global} will be  proved   for  $n$ even  as  Theorem \ref{theo:main-par}, and $n$ odd  as Theorem \ref{theo:main-impar}. 
The main  structural  statements  for  their  proof, namely Lemmas \ref{periodicity} and  \ref{lemma:sequence-even}   for the  even case, and Lemma \ref{lemma:no-ptorsion-odd} are  of  different  nature and  therefore  stated and proved separately.

Within the class  of  crystallographic  groups
$$ 1\to \mathbb{Z}^{n}\to \Gamma= \mathbb{Z}^{n}\rtimes_{\rho}\mathbb{Z}/m \to \mathbb{Z}/m\to 1,$$
there  exist groups for  which the  unstable  Gromov-Lawson-Rosenberg  conjecture is  known to  be  true, namely 
\begin{itemize}

\item The number $m$ is  prime  and  the  action  is  free  outside  of  the origin,  according  to \cite{DL13}.  
\item  The groups  adressed  in section \ref{sec:connective}.
\end{itemize}

 On the  other  hand  side,  the  group $\mathbb{Z}^{4}\times \mathbb{Z}/3 $ belongs   to  the  family  of  central extensions  
  
$$ 1\to \mathbb{Z}^{n}\to \Gamma= \mathbb{Z}^{n}\rtimes_{\rho}\mathbb{Z}/m \to \mathbb{Z}/m\to 1,$$
for  a  representation $\rho: \mathbb{Z}/3\to  GL_{5}(\mathbb{Z})$ whose  action  on $\mathbb{R}^{5}$ is  trivial. When we localize at a prime number $p$, we have a complete determination of the group cohomology of $\Gamma$ in terms of  the  decomposition  of a finite index submodule of $\mathbb{Z}^n$ as a  $\mathbb{Z}[\mathbb{Z}/m]$-module where the summands are  the irreducible  representations $\mathbb{Z}$ (trivial representation), $\mathbb{I}[\mathbb{Z}/n]$ (augmentation ideal), and $\mathbb{Z}[\mathbb{Z}/m]$. The  information related  to  these  decompositions is  one  of  the  ingredients for  the  negative results on the  Gromov-Lawson-Rosenberg conjecture. 

The  second  ingredient  will  be  the verification, using  the computations of group cohomology  by \cite{SV},  that  there  exists a  family  of  crystallographic  groups  for  which the method introduced  by  Schick  in \cite{schick} applies. The  following is our main  negative result,  which  will be  proved as Theorem \ref{theo:counterexamples}. 
\begin{theorem*}
  Suppose $m$ is square-free. Let $\Z^n$ be a $\Z/m$-module, and suppose that there exists a prime $p\mid m$ such that if we consider the $(r,s,t)$ decomposition of $M$ viewed as a $\Z/p$-module, where $r\geq 4$, and $s+t\geq1$ then $\Z^n\rtimes\Z/m$ is a counter-example for the unstable Gromov-Lawson-Rosenberg conjecture.  
\end{theorem*}

\subsection{Aknowledgements}
The first named author aknowledges support  of DGAPA-PAPIIT Grant IN101423. Both  authors  thank the  exchange program CIC-UNAM.

\section{Connective $ko$-homology  of $B\Gamma$. }\label{sec:connective}

Recall that $\Gamma$ fits in an extension
\begin{equation}\label{extension}
1\to\Z^n\to\Gamma\xrightarrow{\pi}\Z/m\to 1,
\end{equation}
Assume that condition \ref{condition:positive} holds. 
Then, as  a  consequence  of  \cite{LW},  we  obtain the  following  result.

\begin{theorem}\label{diagram:main}
   There is a commutative diagram with exact rows    $$\xymatrix{\bigoplus_{(N)\in\mathcal{N}}\widetilde{ko}_r(BN)\ar[r]&ko_r(B\Gamma)\ar[r]^{\beta}\ar[d]^{A\circ p_{B\Gamma}}&ko_r(\underline{B}\Gamma)\ar[d]^{p_{\underline{B}\Gamma}}\\& KO_r(C_r^*(\Gamma;\dbR))\ar[r]&KO_r(\underline{B}\Gamma)}$$
  
where the bottom map is the composite of the inverse of the Baum–Connes map and the map $KO_m^\Gamma(\underline{E}\Gamma)\to KO_m(B\Gamma)$ is  the   induction map with respect to the  homomorphism $\Gamma\to 1$. 
\end{theorem}
\begin{proof}
It is a consequence of the cellular pushout relating $B\Gamma$ and $\underline{B}\Gamma$ when the group $\Gamma$ satisfies conditions M and NM.
\end{proof}
Now suppose that $f_M:M\to B\Gamma$ is a classifying map of $M$ and  $$\alpha(M)=A\circ p_{B\Gamma}(D[f_M])=0.$$ 

First, by the commutativity of the above diagram, we have $p_{\underline{B}\Gamma}\circ\beta(D[f_M])=0$. Now we will analyze $\ker(p_{\underline{B}\Gamma}).$

 \begin{lemma}\label{periodicity}$\ker(p_{\underline{B}\Gamma})$ only contains $m$-torsion. 
\end{lemma}
\begin{proof}
Let $p$ be a prime dividing $m$ such that $m=p^sm'$ with $(p,m')=1$ and let $r$ be the product of the primes dividing $m$. By Lemma \ref{coinvariants} the quotient map $\underline{B}\Gamma\to \underline{B}(\Gamma/\dbZ/{p^s})$  induces an isomorphism $$ko_*(\underline{B}\Gamma)_{\dbZ/{p^s}}\otimes\dbZ[1/p]\to ko_*(\underline{B}(\Gamma/\dbZ/{p^s}))\otimes\dbZ[1/p].$$
We have a commutative diagram

$$\xymatrix{ko_*(\underline{B}\Gamma)_{\Z/{p^s}}\otimes\Z[1/p]\ar[r]^-{\cong}\ar[d]^{p_{\underline{B}\Gamma}}&ko_*(\underline{B}(\Gamma/\dbZ/{p^s}))\otimes\Z[1/p]\ar[d]^{p_{\underline{B}(\Gamma/\dbZ/{p^s})}}\\ KO_*(\underline{B}\Gamma)_{\Z/p^s}\otimes\Z[1/p]\ar[r]^-{\cong}&KO_*(\underline{B}(\Gamma/\dbZ/{p^s}))\otimes\Z[1/p]}$$
Then it is enough to prove that $$ko_*(\underline{B}(\Gamma/\dbZ/{p^s}))_{\Z/p^s}\otimes\dbZ[1/p]\to KO_*(\underline{B}(\Gamma/\dbZ/{p^s}))_{\Z/p^s}\otimes\dbZ[1/p]$$ only contains $m/p^s $-torsion. 

The commutativity  of  the  previous  diagram follows  from  the  fact  that  periodicity is a transformation  of  homology  theories, and  that the  composition  of a  homology  theory  with the  formation  of  coinvariants  is  a  homology  theory. 

Proceeding inductively over all primes dividing $m$ it is enough to prove that $$ko_*(B\dbZ^n)_{\Z/m}\otimes\Z[1/m]\xrightarrow{p_{B\Z^n}} KO_*(B\dbZ^n)_{\Z/m}\otimes\Z[1/m]$$ is injective. But it is a consequence of the injectivity of the maps on coefficients from connective real K-theory and periodic real K-theory as in noted by Davis and L\"uck in the proof of Thm. 0.7 in \cite{DL13}. Then we find that $\ker(p_{\underline{B}\Gamma})$ only contains $m$-torsion.
\end{proof}
Now we need to recall a couple of results that we will need in the following.
\begin{lemma}\label{coinvariants}
Let $p$ be a prime number and let $G$ be a $p$-group.  For any $G$-CW complex $X$  with quotient map $\pi:X\to G\setminus X$ and any homology theory $\calH_*(-)$, the induced map $\pi_r:\left(\calH_r(X)\otimes\Z[1/p]\right)_G\to \calH_r(G\setminus X)\otimes\Z[1/p]$ is an isomorphism for all $r\in \Z$.
\end{lemma}
\begin{proof}
    The argument is similar to the one given for Prop. A.4. in \cite{DL13}. Given a $G$-CW-complex $X$, we get a natural transformation 
    $$j_*:\left(\calH_m(X)\otimes\Z[1/p]\right)_G\to \calH_m(G\setminus X)\otimes\Z[1/p].$$Both sides are $G$-homology theories and moreover $j_*$ is an isomorphism when $X=G/H$ we have that $j_*$ is an isomorphism for ever $X$.
\end{proof}

We aso need the following result of L\"uck-Weiermann. 
     \begin{theorem}[Corollary 2.8 in \cite{LW}]\label{LW}
     Let $\calF\subseteq \calG$ be families of subgroups of a group $\Gamma$ such that every element in $\calG-\calF$ is contained in a unique maximal element in $\calG-\calF$. Let $\calM$ be a complete system of representatives of the conjugacy classes
of subgroups in $\calG-\calF$ which are maximal in $\calG-\calF$. Let $\mathcal{SUB}(M)$ be the family of subgroups of $M$. Then, there is a cellular $\Gamma$-pushout
\begin{equation}\label{pushout}
\xymatrix{\bigsqcup_{M\in\calM} B_{\calF\cap N_\Gamma M} (N_\Gamma M)\ar[r]^-i\ar[d]^{\lambda}&B_{\calF}(\Gamma)\ar[d]\\
\bigsqcup_{M\in\calM} B_{\mathcal{SUB}(M)\cup(\calF\cap N_\Gamma M)} (N_\Gamma M)\ar[r]&X
}
\end{equation}
such that $X$ is a model for $B_{\calG}(\Gamma)$.
 \end{theorem}

For a group $\Gamma$ given by (\ref{extension}) let us denote by $\Gamma_p=\pi^{-1}(\Z/p^s)$, where $p^s$  is the highest power  of  $p$ which divides $m$. For any subgroup $G\subseteq \Z/m$ we denote by  $G_p$ the  subgroup of $\Gamma$ defined  as   $G_{p}=G/(\Z/p^s\cap G).$ 

For a homology theory $\calH_*(-)$ denote by $\calH_*(-)_{(p)}$ to the localization of $\calH_*$ at $p$, that is, $\calH_*(-)\otimes\Z_{(p)}$, where $\Z_{(p)}$ is the ring of integers localized at the prime ideal $(p)$.

\begin{lemma}\label{p-local}
    For any (reduced or unreduced) homology theory  $\calH_*(-)$ and $p\mid m$, we have that
    \begin{enumerate}[(i)]
    \item $\calH_*(B\Gamma)_{(p)}\cong \left(\calH_*(B(\Gamma_p))_{(p)}\right)_{G_p}$
    \item $\calH_*(\underline{B}\Gamma)_{(p)}\cong \left(\calH_*(\underline{B}(\Gamma_p))_{(p)}\right)_{G_p}.$
    \end{enumerate}
\end{lemma}

\begin{proof}
\begin{enumerate}[(i)]
    \item 
The extension   induces a fibration of classifying spaces.
\begin{equation*}
    B\Gamma_p\to B\Gamma\to BG_p
\end{equation*}

Consider the Leray-Serre spectral sequence associated to this fibration, for the homology theory $\calH_*(-)_{(p)}$. By \cite{milnor} this spectral sequence converges to $\calH_*(B\Gamma)_{(p)}$ with second page:
\begin{equation*}
E^2_{\alpha,\beta}=H_\alpha\left(G_p;\calH_\beta(B\Gamma_p)_{(p)}\right)=
    \begin{cases}
\left(\calH_\beta(B\Gamma_p)_{(p)}\right)_{G_p}&\alpha=0\\
    0&\alpha\neq0.
        \end{cases}
    \end{equation*}
This implies that the sequence collapses without extension problems, so we have proved the first statement.
\item

The idea of the proof is similar to that of Lemma 3.1 in \cite{SV2}. Let $\calF$ be a family of subgroups of $\Gamma$, we say that $\Gamma$ satifies the \emph{$(p^s,\calF)$-condition} if for every homology theory $\calH_*(-)$, the induction 
\begin{equation*}\label{coinvariant}
    \left(\calH_*(B_{\calF\cap\Gamma_p}\Gamma_p)_{(p)}\right)_{G_p}\xrightarrow{Ind_{\Gamma_p}^{\Gamma}} \calH_*(B_\calF\Gamma)_{(p)},
    \end{equation*}
    is an isomorphism. Here, $B_\calF\Gamma$ denotes the orbit space of the classifying space of $\Gamma$ with respect to $\calF$.

    Suppose that $m$ is the product of powers of distinct primes $p_1^{s_1}\cdots p_r^{s_r}$. Let $$\calF_i =\{ H\subset \Gamma\mid H \text{ is finite and }|H|\text{ divides }p_1^{s_1}\cdots p_i^{s_i} \}.$$

    Now we will show that by induction on $i$ that $\Gamma$ satisfies the $(p^s,\calF_i)$-condition, for every $i$.

    For $i=0$, it is the statement (i). 
\end{enumerate}
Now we will construct a model for 
$B_{\mathcal{SUB}(M)\cup(\calF_{i}\cap N_\Gamma M)} (N_\Gamma M)$.

We have a pushout
\begin{equation}\label{union}
\xymatrix{B_{\mathcal{SUB}(M)\cap(\calF_i\cap N_\Gamma M)} (N_\Gamma M)\ar[r]^-i\ar[d]^{\lambda}&B_{\calF_i\cap N_\Gamma M}(N_\Gamma M)\ar[d]\\
B_{\mathcal{SUB}(M)} (N_\Gamma M)\ar[r]&B_{\mathcal{SUB}(M)\cup(\calF_i \cap N_\Gamma M)
}(N_\Gamma M)}\end{equation}
In both cases, families $\mathcal{SUB}(M)$ and $\mathcal{SUB}(M)\cap\calF_i\cap N_\Gamma M$ have a maximal element, $M$ and $M_{i+1}=M/(\Z/p_{i+1}^{s_{i+1}})$ respectively, then $E(W_\Gamma M)$ with the action induced by the quotient map is a model for $E_{\mathcal{SUB}(M)}N_\Gamma M$, then $$B_{\mathcal{SUB}(M)}N_\Gamma M=BW_\Gamma M$$ and $E(N_\Gamma M/(M_{i+1}))$ with the action induced by the quotient map is a model for 
\\$E_{\mathcal{SUB}(M)\cap\calF_i\cap N_\Gamma M}N_\Gamma M$, and  hence 
$$B_{\mathcal{SUB}(M)\cap\calF_i\cap N_\Gamma M}N_\Gamma M=B(N_\Gamma M/M_{i+1}).$$

Notice that both groups $W_\Gamma M$ and $N_\Gamma M/M_{i+1}$ are finite subgroups of $\Z/m$. We have 
\begin{align*}\calH_*(B_{\mathcal{SUB}(M)}N_\Gamma M)_{(p)}&\cong\calH_*(BW_\Gamma M)_{(p)}\\&\cong\left(\calH_*(B(W_\Gamma M\cap \Gamma_p/M))_{(p)}\right)_{G_p}\text{ by (i) }\\&\cong\left(\calH_*(B_{\mathcal{SUB}(M)}(N_\Gamma M\cap \Gamma_p))_{(p)}\right)_{G_p}.\end{align*} and in a similar way

\begin{align*}\calH_*(B_{\mathcal{SUB}(M)\cap\calF_i\cap N_\Gamma M}N_\Gamma M)_{(p)}&\cong\left(\calH_*(B_{\mathcal{SUB}(M)\cap\calF_i\cap N_\Gamma M}(N_\Gamma M\cap \Gamma_p))_{(p)}\right)_{G_p}.\end{align*}
 In others words, $N_\Gamma 
 M$ satisfies $(p^s,\mathcal{SUB}(M))$-condition and $(p^s,\mathcal{SUB}(M)\cap\calF_i\cap N_\Gamma M)$-condition. By the five lemma applied to the morphism of Mayer-Vietoris sequences given by restrict the pushout (\ref{union}) to $\Gamma_p$ we get that $N_\Gamma M$ satisfies the $(p^s,\mathcal{SUB}(M)\cup(\calF_i\cap N_\Gamma M))$-condition.

Finally again by the five lemma applied to the morphism of Mayer-Vietoris sequences given by restriction to $\Gamma_p$ of the pushout (\ref{pushout}) associated to families $\calF_i\subseteq\calF_{i+1}$ we get $\Gamma$ satisfies the $(p^s,\calF_i)$-condition for every $i$ and for $i=r$ we obtain (ii). 

\end{proof}

The following  result  was  proved  as Lemma 4.4  in \cite{LL12}.

\begin{lemma}\label{lemma:sequence-even}
 In the  exact sequence 
 $$0\longrightarrow H^{2r}(\underline{B}(\Z ^{n}\rtimes \Z/p^{s})) \overset{\bar{f}_{*}}{\longrightarrow} H^{2r}(\Z^{n}\rtimes \Z /p^{s})\overset{\varphi^{2r}}{\longrightarrow} \underset{P\in \mathcal{P}}{\bigoplus} H^{2r}(BP),$$
 \begin{itemize}
     \item The homomorphism $\varphi^{2r}$ has torsionfree kernel. 
     \item The abelian  group $H^{2r}(\underline{B}(\Z ^{n}\rtimes \Z/p^{s})) $ is  finitely  generated  and torsionfree. 
    
 \end{itemize}

\end{lemma}
Applying Lemma \ref{p-local} and the Universal Coefficient Theorem, we obtain the following result. 
\begin{corollary}\label{lemma:integral-homology-odd}
The abelian groups  
$$H_{2r+1}(\underbar{B} \Gamma , \Z)$$
are  finitely  generated  and  torsion-free. 
\end{corollary}

 \section{Proof  of  the positive result in the  even dimensional case.  }
In this  section  we  will prove  the  even  version  of  Theorem \ref{theo:global}. 

\begin{theorem}\label{theo:main-par}
 Let $M^n$ be an $n$-dimensional smooth spin manifold, where $n\geq 5$ is  even, and  with fundamental group isomorphic to $\Gamma$.  Denote  by $f_{M}: M\to B\Gamma $ the  classifying map for  the  fundamental group. Assume  that $\alpha(M) =0$. Then $M$ admits a  metric of  positive scalar curvature.
\end{theorem}

  The proof of this theorem requires some work.

\begin{lemma}\label{label:ko-even}
    For $*$ even, $ko_*(B\Gamma)$ does not contain $m$-torsion
\end{lemma}
\begin{proof}
  By \cref{p-local} it is enough to prove that for any odd prime $p$ dividing $m$, we have that $ko_*(B(\Z^n\rtimes\Z/p^s))_{G_p}$ does not contain $p^s$-torsion.

  Consider the Atiyah-Hirzebruch- Leray-Serre spectral sequence associated to the extension
  $$0\to\Z^n\to\Z^n\rtimes\Z/p^s\to\Z/p^s\to0,$$
   coverging to $ko_*(B(\Z^n\rtimes\Z/p^s))$. The second page is given by
   $$E^2_{p,q}=H_i(\Z/p^s;ko_j(B\Z^n)).$$
Let us first compute $ko_j(B\Z^n))$. By Lemma 5.3 in \cite{DL13} we have isomorphisms of abelian groups
\begin{align}\label{ko:2}
    ko_j(B\Z^n)\otimes\Z[1/2]&\cong\bigoplus_{l=0}^nH_l(\Z^n)\otimes ko_j(*)\otimes\Z[1/2]\\
    \label{ko:p}ko_j(B\Z^n)\otimes\Z_{(2)}&\cong\bigoplus_{l=0}^nH_l(\Z^n)\otimes ko_j(*)_{(2)}.
\end{align}
Now we will prove that both are  actually isomorphisms of $\Z[\Z/p^s]$-modules. Note that $ko_j(*)\otimes\Z[1/2]$ is torsion free, then the Chern character shows that the map (\ref{ko:2}) is an isomorphism of $\Z[\Z/p^s]$-modules. The isomorphism (\ref{ko:p}) implies that the Atiyah-Hirzebruch spectral sequence with second page$$E^2_{i,j}=H_i(\Z^n;ko_j(*)_{(2)})$$converging to $ko_{i+j}(B\Z^n)_{(2)}$ collapses. On the other hand, this spectral sequence is natural with respect to $\Z/p^s$-module structure of $\Z^n$.  We have a filtration of $\Z_{(2)}[\Z/p^s]$-modules
$$ko_r(B\Z^n)_{(2)}=F_{0,r}\supseteq\ldots\supseteq F_{r,0}\supseteq F_{r+1,-1}=0$$
and exact sequences of $\Z_{(2)}[\Z/p^s]$-modules
$$0\to F_{i+1,r-i-1}\to F_{i,r-i}\to H_i(\Z^n)\otimes ko_{r-i}(*)_{(2)}\to0$$
It is enough to prove that the above exact sequence splits as $\Z_{(2)}[\Z/p^s]$-modules.

Suppose first that $r-i\equiv 3,5,6,7$ (mod 8), then $ko_{r-i}(*)=0$, and so the exact sequence splits trivially. 

If $r-i\equiv 0,4$ (mod 8), then $ko_{r-i}(*)\cong\Z$. On the other hand, $H_i(\Z^n)\otimes ko_{r-i}(*)_{(2)}$ is a finite generated $\Z_{(2)}[\Z/p^s]$-module that is torsion free as $\Z_{(2)}$-module, as the action is free outside the origin, the norm element $x$ of $\Z_{(2)}[\Z/p^s]$ acts by zero, then $H_i(\Z^n)\otimes ko_{m-i}(*)_{(2)}$ can be considered as a $\Z_{(2)}[\Z/p^s]/\langle x\rangle$-module. But $\Z_{(2)}[\Z/p^s]/x$ is a Dedekind domain (being isomorphic to the $p^s$-ciclotomic ring), then $H_i(\Z^n)\otimes ko_{m-i}(*)_{(2)}$ is $\Z_{(2)}[\Z/p^s]/x$-projective, and hence it is $\Z_{(2)}[\Z/p^s]$-projective, and  hence the sequence splits.

Finally, if $r-i\equiv 1,2$ (mod 8). As the Atiyah-Hirzebruch spectral sequence collapses, we have an splitting of abelian groups
$$s\colon H_i(\Z^n)\otimes ko_{r-i}(*)_{(2)}\to F_{i,r-i}$$ such that $\pi\circ s= id$. Define
\begin{align*}
    \widetilde{s}:H_i(\Z^n)\otimes ko_{r-i}(*)_{(2)}&\to F_{i,r-i}\\
    y&\mapsto \sum_{g\in\Z/p^s}g\cdot s(g^{-1}x).
\end{align*}
$\widetilde{s}$ is a homomorphism of $\Z_{(2)}[\Z/p^s]$-modules and as $p$ is odd $\pi\circ\widetilde{s}$ is multiplication by $p^s$, but it is the identity because $ko_{m-i}(*)$ is isomorphic to $\Z/2$.

Then the maps (\ref{ko:2}) and (\ref{ko:p}) are isomorphisms of $\Z[\Z/p^s]$-modules. Using that we have $$\widehat{H}^{i+1}(\Z/p^s;ko_j(B\Z^n))\cong \bigoplus_{l}\widehat{H}^{i+1}(\Z/p^s;H_{j-4l}(\Z^n)).$$Then using the universal coefficient theorem and Theorem 3.2 in \cite{LL12} we obtain
$$\widehat{H}^{i+1}(\Z/p^s;ko_j(B\Z^n))=0 \text{ if } i+j \text{ is even .}$$

In particular this implies that the canonical map $$E^2_{0,2r}=ko_{2r}(B\Z^n)_{\Z/p^s}\to ko_{2r}(B\Z^n)^{\Z/p^s}$$is injective, because $\widehat{H}^{-1}(\Z/p^s;ko_{2r}(B\Z^n))=0$. Then $ko_{2r}(B\Z^n)_{\Z/p^s}$ does not contain $p^s$-torsion. On the other hand, $E^2_{i,j}$  is zero if $i+j$ is even and positive, then

$$ko_{2r}(B\Z^n)_{\Z/p^s}=E^2_{0,2r}=E^\infty_{0,2r}\cong ko_{2r}(B(\Z^n\rtimes\Z/p^s)),$$

then $ko_{2r}(B(\Z^n\rtimes\Z/p^s))$ does not contain $p^s$-torsion and we conclude that $ko_{*}(B\Gamma)$ does not contains $m$-torsion if $*$ is even.
\end{proof}
\begin{proof}(\cref{theo:main-par})
Notice that from the Atiyah-Hirzebruch spectral sequence we obtain, for even indices, $\widetilde{ko}_*(B\Z/p^s)=0$. Then from  \cref{diagram:main}
  we conclude that for even degrees $\beta$ is injective   and from Lemma \ref{periodicity} we get $D[f_M]=0$, by Prop. 12.1 in \cite{DL13} we get that $M$ admits a metric with positive scalar curvature.

\end{proof}

\section{Proof of the positive result  in the  odd dimensional case. }
In this section we will prove the odd version of Theorem \ref{theo:global}
\begin{theorem}\label{theo:main-par}
 Let $M^n$ be an $n$-dimensional smooth spin manifold, where $n\geq 5$ is  odd, and  with fundamental group isomorphic to $\Gamma$.  Denote  by $f_{M}: M\to B\Gamma $ the  classifying map for  the  fundamental group. Assume  that $\alpha(M) =0$. Then $M$ admits a  metric of  positive scalar curvature.
\end{theorem}

As in the previous section the proof of this theorem requires some lemmas.

\begin{lemma}\label{lemma:no-ptorsion-odd}
Let $r$ be  a  natural number. Then, the  $ko$-homology  group
$$ko_{2r+1}(\underbar{B}\Gamma) $$
does not contain $p$-torsion for $p\neq 2$. 
\end{lemma}
\begin{proof}
Recall the  Atiyah-Hirzebruch spectral sequence converging  to  the  $ko$-homology  groups localized at  $p$, $ko_{*}(\underbar{B}\Gamma)_{(p)}$   with $E^{2}$-term 
$$ E^{2}_{i,j}= H_{i}(\underbar{B}\Gamma,{ ko_{j}(\ast)}_{(p)}). $$

The relevant  elements on the  $E^{2}$-term  of the   Atiyah-Hirzebruch spectral sequence for  the  computation  of  $ ko_{2r+1}(\underbar{B}\Gamma)$ are those  for  which  $i+j$ is odd. 
Let  us  distinguish the  following  cases: 
\begin{itemize}
\item Assume $i$ is odd, and  notice  that because of \ref{lemma:integral-homology-odd}, the homology  groups  $H_{i}(\underbar{B}\Gamma, \Z)$ are  torsionfree abelian.
Since $i+j$  is  odd, it  follows  that $j$ is  even. Let  us  analyze  first  the  case  where $j\equiv 0 \, {\rm or}\, \equiv 4$ modulo 8. In  both  cases, $ko_{j}(\ast)$ is  free  abelian  of  rank  one,  
and ${ko_{j}(\ast)}_{(p)}$ is a  free $\Z_{(p)}$-module  of  rank  one.  For $j\equiv 2$   modulo $8$, the  group  $ko_{j}(\ast)_{(p)}$ is zero, and   $ko_{j}(\ast)\cong 0 $ for  $j\equiv 6$ modulo 8. 
\item Assume  that $i$ is  even. Then $j$ is  odd. If $j\equiv 3,\, 5,\, \, {\rm or} \, 7$ modulo  $8$, then $ko_{j}(\ast)\cong 0$. If $j\equiv 1$ modulo $8$, then ${ko_{j}}_{(p)}=0$. 
\end{itemize}
In  either  case, we see  that  $E_{i,j}^{2}$ is  either  zero  or  a free $\Z_{(p)}$-module  of finite rank.  Moreover,  because  of  the  rational  triviality  of  the  differentials  of  the Atiyah-Hirzebruch  spectral sequence for  $ko_{2r+1}(\underbar{B}\Gamma)_{(p)}$,  the  spectral sequence  collapses  without differentials  and  extension problems,  converging  to   free $\Z_{(p)}$-modules. 

We analyze  now the Atiyah-Hirzebruch spectral sequence  for $ko_{2r+1}(\underbar{B}\Gamma)[\frac{1}{p}]$. By Lemma \ref{coinvariants}, these  groups  are  isomorphic  to  

$$\left(ko_{2r+1}(B\Z^{n})\left[\frac{1}{p}\right]\right)_{\Z/p^{s}},$$

By  (\ref{ko:p}), there are positive integers $r_i$, such that, these  groups  are  isomorphic  to  
$$ \underset{i}{\bigoplus} ko_{2r+1-i}(\ast)^{r_{i}}\left[\frac{1}{p}\right].
$$
If $i$ is even, the groups $ko_{2r+1-i}$ are   either  zero or  two-torsion, and  hence 
$ko_{2r+1-i}(\ast)^{r_{i}}\left[\frac{1}{p}\right]=0$.

If  $i$ is  odd,  the  sum $2r+1-i$ is even. Put  $i=2l_{i}+1$ for $l_{i}$ a  natural  number  or  zero.
if $r- l_{i}$ is  even, then  $2(r-l_{i})+2 \equiv \, 2\,{\rm mod} \, 4$, and  $ko_{2r+1-i}(*)\left[\frac{1}{p}\right]=0$.  If $r-l_{i}$ is odd, $2(r-l_i)+2$ is  divisible  by  four,  and   hence $ko_{2r+1-i}(*)\left[\frac{1}{p}\right]=ko_{2(r-l_{i})+2}(*)\left[\frac{1}{p}\right]= \Z\left[\frac{1}{p}\right]$.

In  either  case,  we  conclude  that  the  groups  are  zero  or  free $\Z\left[\frac{1}{p}\right]$-modules. 
\end{proof}
The  following  two  results  concern   computations  of  the  spin bordism  groups of  the  classifying  space, and  the  real $K$-theory  of  the  real group $C^{*}$-algebra of   the  finite  group $\Z/p^{s}$. 

\begin{lemma}\label{lemma:bordism-surjective-odd}
    For any $m$ odd, the map 
$$\widetilde{D}:\widetilde{\Omega}_*^{Spin}(B\dbZ/m)\to \widetilde{ko}_*(B\Z/m)$$is surjective.
\end{lemma}
\begin{proof}
By Lemma \ref{p-local} (i), it is enough to prove that $$\widetilde{D}:\widetilde{\Omega}^{Spin}_{*}(B\Z/p^s)\to \widetilde{ko}_*(B\Z/p^s)$$ is surjective for any odd prime $p$.

Let  $ M $ be  a $\mathbb{Z}/p^s$-module.  By  the  standard  resolution  of  $\mathbb{Z}$ as  a  trivial $\mathbb{Z}[\mathbb{Z}/p^s]$-  module, for any  $i\geq 1$,    the  localization  of  the  homology  groups  with  coefficients  in $M$, 
 $$ H_{i}(\mathbb{Z}/p^s , M)\left[\frac{1}{p}\right]= 0$$
holds.

It  follows  that for $i\geq 1$  the   maps 
\begin{equation}\label{iso:localization}H_{i}(\mathbb{Z}/p^s, M)\to H_{i}(\mathbb{Z}/p^s, M)_{(p)} \to H_{i}(B\mathbb{Z}/ p^s; M_{(p)})\end{equation}

are  all  isomorphisms.  Consider the Atiyah-Hirzebruch spectral sequences converging to $\Omega_*^{Spin}(B\Z/p^s)$, $\Omega_*^{Spin}(B\Z/p^s)_{(p)}$, $ko_*(B\Z/p^s)$ and $ko_*(B\Z/p^s)_{(p)}$. By  the  comparison lemma for spectral sequences \cite[Theorem 5.2.12]{Weibel} and isomorphism (\ref{iso:localization}), we have a commutative diagram

$$\xymatrix{\widetilde{\Omega}_m(B\Z/p^s)\ar[d]^{\cong}\ar[r]^{\widetilde{D}}&\widetilde{ko}_m(B\Z/p^s)\ar[d]^{\cong}\\\widetilde{\Omega}_m(B\Z/p^s)_{(p)}\ar[r]^{\widetilde{D}_{(p)}}&\widetilde{ko}_m(B\Z/p^s)_{(p)}}$$


Then it is enough to prove the surjectivity of $\widetilde{D}_{(p)}$



The  Atiyah-Hirzebruch spectral sequence  for   computing $p$-local Spin  bordism  and  $ko$-homology collapse  at  the  $E^{2}$ term,  yielding  isomorphisms 
$$E_{i,j}^{\infty}= \tilde{H}_{i}(\mathbb{Z}/p^s) \otimes \left(\Omega_{j}^{\rm Spin}\right) _{(p)}$$

$$ E_{i,j}^{\infty}= \tilde{H}_{i}(\mathbb{Z}/p^s)\otimes{ \left({\rm ko}_{j}\right)}_{(p)}$$

Taking  a  look  at the  map  on  the $(p)$-localized coefficients 
$$ D_{(p)}: \left(\Omega_{j}^{\rm Spin}\right)_{(p)}\to \left({\rm ko}_{j}\right)_{(p)},$$
which  are non -zero  only  for  $j$ a  multiple  of  four,  we  see that  the map  is  surjective  at  the level  of  coefficients,  and  thus  surjective  at  the  $E^{\infty}$- term.

\end{proof}

We  now  recall  the  following  consequence  of Theorem 9.4 in page 415  of \cite{DL13}. 
\begin{lemma}\label{lemma:realktheory-finitecyclic}
Let $p$ be an  odd prime number. The real  $K$-theory  of  the  real group $C^{*}$-algebra  for  the cyclic  group $\Z/p^{s}$ is  as  follows:

\begin{center}
\begin{tabular}{||c|c||}
\hline
      $KO_{0}(\mathbb{R}[\Z/p^{s}])$  & $\Z^{1+\frac{p^s-1}{2}}$   \\ \hline 
      $KO_{1}(\mathbb{R}[\Z/p^{s}])$ &   $\Z/2$ \\  \hline
      $KO_{2}(\mathbb{R}[\Z/p^{s}])$ &   $\Z/2\oplus\Z^{\frac{p^{s}-1}{2}}$ \\  \hline
      $KO_{3}(\mathbb{R}[\Z/p^{s}])$ &   $0$ \\  \hline
      $KO_{4} (\mathbb{R}[\Z/p^{s}])$&    $\Z^{1+\frac{p^{s}-1}{2}}$\\  \hline
      $KO_{5} (\mathbb{R}[\Z/p^{s}])$&    $0$\\  \hline
      $KO_{6} (\mathbb{R}[\Z/p^{s}])$&    $\Z^{\frac{p^{s}-1}{2}}$\\  \hline
      $KO_{7}(\mathbb{R}[\Z/p^{s}])$ &    $0$\\
      \hline
\end{tabular}
\end{center}
\end{lemma}

\begin{theorem}\label{theo:main-impar}
Let $M^{2r+1}$ be a $2r+1$-dimensional smooth spin manifold, where $r\geq 2$ is  odd, and  with fundamental group isomorphic to $\Gamma$.  Denote  by $f_{M}: M\to B\Gamma $ the  classifying map for  the  fundamental group. Assume  that $\alpha(M) =0$. Then $M$ admits a  metric of  positive scalar curvature. 
\end{theorem}

\begin{proof}
    Consider  the diagram  from Theorem \ref{diagram:main} together  with  the additional left column given by $A\circ p_{BN}$.

$$\xymatrix{\bigoplus_{(N)\in\mathcal{N}}\widetilde{ko}_{2r+1}(BN)\ar[r]^-{\varphi_{2r+1}} \ar[d]^-{A\circ p_{BN}}&ko_{2r+1}(B\Gamma)\ar[r]^{\beta}\ar[d]^{A\circ p_{B\Gamma}}&ko_{2r+1}(\underline{B}\Gamma)\ar[d]^{p_{\underline{B}\Gamma}}\\ \underset{N\in \mathcal{(N})}{\bigoplus}\widetilde{KO}_{2r+1}(\mathbb{R}[\Z/p^{s}])& KO_{2r+1}(C_r^*(\Gamma;\dbR))\ar[r]&KO_{2r+1}(\underline{B}\Gamma)}$$

Recall  that  from  \ref{lemma:no-ptorsion-odd}, the  group $ko_{2r+1}(\underbar{B}\Gamma)$ does  not  contain $m$-torsion,  and  from  \ref{periodicity}, $\ker p_{\underbar{B}\Gamma}$ only  consists  of  $m$-torsion. 
It follows  that  $\beta(D_{[F_{M}]})=0$. By  the  exactness of  the upper  line, we  can  find $X_{N}\in \widetilde{k}o_{2r+1}(BN)$ in the  preimage  under  the  map  $\varphi_{2r+1}$. 
By  the  surjectivity  lemma \ref{lemma:bordism-surjective-odd}, we  can  find classes $[M_{N}, F_{N}]\in \Omega_{2r+1}^{Spin}(BN)$ such that  $D[M_{N}\overset{F_{N}}{\to} BN]= X_{N}$. By  surgery,  we  can  assume  that $F_{M}$ is 2-connected. 
\

Now, recall  that  by   the  computation of  the  real $K$-theory  of  the  group $C^{*}$ algebra  of  $\Z/p^{s}$ \ref{lemma:realktheory-finitecyclic}, the reduced $KO$-theory  groups  are zero in odd degrees.  By the proof  of  the  Gromov-Lawson -Rosenberg Conjecture  for  groups with  periodic cohomology \cite{BGS},  the classes $[M_{N}, F_{N}]$ admit  representatives $[M_{N}^{+}, F_{M}]$, where  $M_{N}^{+}$  has positive  scalar  curvature.    

We  consider  now the  class 
$$ D[F_{M}:M\to B\Gamma ]= D[ \underset{N\in \mathcal{N}}{\sqcup} M_{N}^{+}\longrightarrow \underset{N\in \mathcal{N}}{\sqcup} BN \longrightarrow B\Gamma  ]$$
  and  notice  that  it  admits  a  representative  of positive  scalar  curvature. This  finishes the  proof  of  theorem \ref{theo:main-impar}.

\end{proof}

\section{Construction  of  Counterexamples}

In this section, we will give a condition on the action of $\Z/m$ on $\Z^n$ that implies that the group $\Z^n\rtimes\Z/m$ is a counterexample for the GLR conjecture.The  assumption  with  be  that  $m$ is odd and free of squares to use results from \cite{SV}.

  Let $p$ be a prime that divides $m$, then by Prop. 4.5 in \cite{SV} there is $\Z/m$-submodule $N\subseteq \Z^n$ of finite index such that this index is coprime with $p$ and such that there is a decomposition of $\Z/m$-modules $$N\cong \Z^r\oplus (\Z[\Z/p])^s\oplus I^t.$$
 Where $\Z$ has the trivial $\Z/p$-action, $I\subseteq \Z[\Z/p]$ is the augmentation ideal and both $I$ and $\Z[\Z/p]$ are endowed with the canonical $\Z/p$-action and moreover $\Z^r$, $(\Z[\Z/p])^s$ and $I^t$ are itself $\Z/m$-modules. We say the $N$ is a $\Z/p$-module of type $(r,s,t)$. By Lemma 5.3 in \cite{SV} we can suppose that $\Z^n$ is a $\Z/m$-module is of type $(r,s,t)$. We have the following result.

\begin{theorem}\label{theo:counterexamples}Suppose $m$ is square-free. Let $\Z^n$ be a $\Z/m$-module, and suppose that there exists a prime $p\mid m$ such that if we consider the $(r,s,t)$ decomposition of $M$ viewed as a $\Z/p$-module, where $r\geq 4$, and $s+t\geq1$ then $\Z^n\rtimes\Z/m$ is a counter-example for the GLR conjecture.
\end{theorem}
\begin{proof}
 Notice that, by Corollary. 4.2 in \cite{AGPP} the Lyndon-Hoschild-Serre spectral sequence for (co-)homology associated to the extension (\ref{extension}) collapses. In particular, 
 $H^1(\Gamma)$ contains as a subgroup $H^1(\Z^r\times \Z/m)$, then let $a_1\ldots,a_4$ be some generators of the torsion-free part of $H^1(\Z^n\times\Z/m)$, viewed as elements in $H^1(\Gamma)$. For each $a_i$, we have the dual elements $\widehat{x}_i\in H_1(\Z^n\times\Z/m)$, let $x_i=\iota_*(\widehat{x}_i)$, where $\iota:\Z^r\times\Z/m\to\Gamma$ is the inclusion

 and let $\widehat{y}$ be an element of $p$-torsion in $H_1(\Z^r\times\Z/m)$. Now we follow the argument in \cite{schick}. Let $w=\iota_*(\widehat{x}_1\times\ldots \widehat{x}_4\times \widehat{y})\in H_5(\Gamma)$. Let us prove that $a_1\cap(a_2\cap(a_3\cap  w))\neq 0$.

 First note that by the Kunneth formula applied to the decomposition as $(r,s,t)$-modules, the map $\iota_*:H_5(\Z^r\times\Z/m)\to H_5(\Gamma)$ is injective. On the other hand, by the naturality of the cap product we have $$\iota_*(a_3\cap (\widehat{x}_1\times\ldots \widehat{x}_4\times \widehat{y}))=\iota_*(\iota^*(a_3)\cap (\widehat{x}_1\times\ldots \widehat{x}_4\times \widehat{y}))=a_3\cap w.$$But $a_3\cap (\widehat{x}_1\times\ldots \widehat{x}_4\times \widehat{y})\neq0,$ and similarly  $a_1\cap(a_2\cap(a_3\cap ( w)))\neq 0$ and now the same argument in Example 2.2 in \cite{schick} applies. Then the group $\Gamma$ is a counterexample for GLR conjecture.
\end{proof}
\bibliographystyle{alpha} 
\bibliography{mybib}
\end{document}